\documentclass[11pt]{amsart}
\usepackage{amsmath}
\usepackage{amssymb}
\usepackage{amscd}

\def\NZQ{\mathbb}               
\def\NN{{\NZQ N}}
\def\QQ{{\NZQ Q}}
\def\ZZ{{\NZQ Z}}
\def\RR{{\NZQ R}}

\def\PP{{\NZQ P}}

\newtheorem{Theorem}{Theorem}[section]
\newtheorem{Lemma}[Theorem]{Lemma}
\newtheorem{Corollary}[Theorem]{Corollary}
\newtheorem{Proposition}[Theorem]{Proposition}
\newtheorem{Remark}[Theorem]{Remark}

\newtheorem{Example}[Theorem]{Example}

\newtheorem{Definition}[Theorem]{Definition}

\let\epsilon\varepsilon
\let\phi=\varphi
\let\kappa=\varkappa

\begin{document}
\title{Examples of multiplicities and mixed multiplicities of filtrations}
\author{Steven Dale Cutkosky}

\thanks{Partially supported by NSF grant DMS-1700046.}

\address{Steven Dale Cutkosky, Department of Mathematics,
University of Missouri, Columbia, MO 65211, USA}
\email{cutkoskys@missouri.edu}

\date{January 11, 2020.}

\dedicatory{This paper is dedicated to Roger and Sylvia Wiegand on the occasion of their combined 150th birthday.}


\begin{abstract} In this paper we construct examples of irrational behavior of multiplicities and mixed multiplicities of divisorial filtrations. The construction makes essential use of  anti-positive intersection products.
\end{abstract}

\maketitle

\section{Introduction}

In this paper, we begin by giving an overview of  the theory of multiplicities and
 mixed multiplicities of  (not necessarily Noetherian) filtrations, including an interpretation of multiplicities and mixed multiplicities of divisorial filtrations as anti-positive intersection multiplicities. Using this interpretation, we construct a resolution of singularities of a normal three dimensional local ring and compute the multiplicities and mixed multiplicities of its divisorial filtrations, showing essentially irrational behavior.

The study of mixed multiplicities of $m_R$-primary ideals in a Noetherian local ring $R$ with maximal ideal $m_R$  was initiated by Bhattacharya \cite{Bh}, Rees  \cite{R} and Teissier  and Risler \cite{T1}. 
In \cite{CSS}  the notion of mixed multiplicities  is extended to arbitrary,  not necessarily Noetherian, filtrations of $R$ by $m_R$-primary ideals. It is shown in \cite{CSS} that many basic theorems for mixed multiplicities of $m_R$-primary ideals are true for filtrations. 

The development of the subject of mixed multiplicities and its connection to Teissier's work on equisingularity \cite{T1} can be found in \cite{GGV}.   A  survey of the theory of  mixed multiplicities of  ideals  can be found in  \cite[Chapter 17]{HS}, including discussion of the results of  the papers \cite{R1} of Rees and \cite{S} of  Swanson, and the theory of Minkowski inequalities of Teissier \cite{T1}, \cite{T2}, Rees and Sharp \cite{RS} and Katz \cite{Ka}.   Later, Katz and Verma \cite{KV}, generalized mixed multiplicities to ideals which are not all $m_R$-primary.

Let $R$ be a Noetherian local ring of dimension $d$ with maximal ideal $m_R$.
Let $\ell_R(M)$ denote the length of an $R$-module $M$.

A filtration $\mathcal I=\{I_n\}_{n\in\NN}$ of  a ring $R$ is a descending chain
$$
R=I_0\supset I_1\supset I_2\supset \cdots
$$
of ideals such that $I_iI_j\subset I_{i+j}$ for all $i,j\in \NN$.  A filtration $\mathcal I=\{I_n\}$ of  a local ring $R$ by $m_R$-primary ideals is a filtration $\mathcal I=\{I_n\}_{n\in\NN}$ of $R$ such that   $I_n$ is $m_R$-primary for $n\ge 1$.
A filtration $\mathcal I=\{I_n\}_{n\in\NN}$ of  a ring $R$ is said to be Noetherian if $\bigoplus_{n\ge 0}I_n$ is a finitely generated $R$-algebra.

The nilradical $N(R)$ of  $R$ is 
$$
N(R)=\{x\in R\mid x^n=0 \mbox{ for some positive integer $n$}\}.
$$
We have that $\dim N(R)=d$ if and only if there exists a minimal prime $P$ of $R$ such that $\dim R/P =d$ and $R_P$ is not reduced.  Let $\hat R$ be the $m_R$-adic completion of $R$.

In \cite[Theorem 1.1]{C2} and  \cite[Theorem 4.2]{C3} we have shown that     the limit 
\begin{equation}\label{I5}
\lim_{n\rightarrow\infty}\frac{\ell_R(R/I_n)}{n^d}
\end{equation}
exists for any filtration  $\mathcal I=\{I_n\}$ of $R$ by $m_R$-primary ideals if and only if $\dim N(\hat R)<d$.
We observe that the condition $\dim N(\hat R)<d$ holds if $R$ is analytically unramified; that is, $\hat R$ is reduced.

The problem of existence of such limits (\ref{I5}) has been considered by Ein, Lazarsfeld and Smith \cite{ELS} and Musta\c{t}\u{a} \cite{Mus}.
When the ring $R$ is a domain and is essentially of finite type over an algebraically closed field $k$ with $R/m_R=k$, Lazarsfeld and Musta\c{t}\u{a} \cite{LM} showed that
the limit exists for all filtrations  of $R$ by $m_R$-primary ideals.  Cutkosky proved it in the complete generality  stated above in \cite{C2} and  \cite{C3}. These proofs use the theory of volumes of cones of Okounkov \cite{Ok}, Kaveh and Khovanskii \cite{KK} and  Lazarsfeld and Musta\c{t}\u{a} \cite{LM}.

 We now impose the necessary condition that 
the dimension of the nilradical $N(\hat R)$ of the completion $\hat R$ of $R$ is less than the dimension of $R$, to insure the existence of limits.
We define the multiplicity of $R$ with respect to a filtration $\mathcal I=\{I_n\}$ of $mR$-primary ideal to be 
$$
e_R(\mathcal I;R)=
\lim_{n\rightarrow \infty}\frac{\ell_R(R/I_n)}{n^d/d!}.
$$
In the case that $\mathcal I=\{I^n\}_{n\in\NN}$ is  the filtration of powers of a fixed $m_R$-primary ideal $I$, the filtration $\mathcal I$ is Noetherian, and we have that
$$
e_R(\mathcal I;R)=e_R(I;R)
$$
is the ordinary multiplicity of $R$ with respect to the ideal $R$ (here we use the notation $e_R(I;R)$ of \cite{HS}).

Mixed multiplicities of filtrations are defined in \cite{CSS}. 
 Let $M$ be a finitely generated $R$-module where $R$ is a $d$-dimensional Noetherian local ring with $\dim N(\hat R)<d$. Let $\mathcal I(1)=\{I(1)_n\},\ldots, \mathcal I(r)=\{I(r)_n\}$ be filtrations of $R$ by $m_R$-primary ideals. 
 In  \cite[Theorem 6.1]{CSS} and  \cite[Theorem 6.6]{CSS}, it is shown that the function
\begin{equation}\label{M2}
P(n_1,\ldots,n_r)=\lim_{m\rightarrow \infty}\frac{\ell_R(M/I(1)_{mn_1}\cdots I(r)_{mn_r}M)}{m^d}
\end{equation}
is equal to a homogeneous polynomial $G(n_1,\ldots,n_r)$ of total degree $d$ with real coefficients for all  $n_1,\ldots,n_r\in\NN$.  

We  define the mixed multiplicities of $M$ from the coefficients of $G$, generalizing the definition of mixed multiplicities for $m_R$-primary ideals. Specifically,   
 we write 
\begin{equation}\label{eqV6}
G(n_1,\ldots,n_r)=\sum_{d_1+\cdots +d_r=d}\frac{1}{d_1!\cdots d_r!}e_R(\mathcal I(1)^{[d_1]},\ldots, \mathcal I(r)^{[d_r]};M)n_1^{d_1}\cdots n_r^{d_r}.
\end{equation}
We say that $e_R(\mathcal I(1)^{[d_1]},\ldots,\mathcal I(r)^{[d_r]};M)$ is the mixed multiplicity of $M$ of type $(d_1,\ldots,d_r)$ with respect to the filtrations $\mathcal I(1),\ldots,\mathcal I(r)$.
Here we are using the notation 
\begin{equation}\label{eqI6}
e_R(\mathcal I(1)^{[d_1]},\ldots, \mathcal I(r)^{[d_r]};M)
\end{equation}
  to be consistent with the classical notation for mixed multiplicities of $M$ with respect to $m_R$-primary ideals from \cite{T1}. The mixed multiplicity of $M$ of type $(d_1,\ldots,d_r)$ with respect to $m_R$-primary ideals $I_1,\ldots,I_r$, denoted by 
  $$
  e_R(I_1^{[d_1]},\ldots,I_r^{[d_r]};M)
  $$ (\cite{T1}, \cite[Definition 17.4.3]{HS}) is equal to the mixed multiplicity 
  $$
  e_R(\mathcal I(1)^{[d_1]},\ldots,\mathcal I(r)^{[d_r]};M),
  $$
   where the Noetherian $I$-adic filtrations $\mathcal I(1),\ldots,\mathcal I(r)$ are defined by 
  $$
  \mathcal I(1)=\{I_1^i\}_{i\in \NN}, \ldots,\mathcal I(r)=\{I_r^i\}_{i\in \NN}.
  $$

We have that 
\begin{equation}\label{eqX31}
e_R(\mathcal I;M)=e_R(\mathcal I^{[d]};M)
\end{equation}
 if $r=1$, and $\mathcal I=\{I_i\}$ is a filtration of  $R$ by $m_R$-primary ideals. Thus
$$
e_R(\mathcal I;M)=\lim_{m\rightarrow \infty}\frac{\ell_R(M/I_mM)}{m^d/d!}.
$$

In \cite{CSS}, it is shown that many classical theorems about mixed multiplicities of $m_R$-primary ideals continue to hold for filtrations. For instance, the four ``Minkowski inequalities'' for mixed multiplicities are proven in \cite[Theorem 6.3]{CSS}. 

\begin{Theorem}(\cite[Theorem 6.3]{CSS}, Minkowski Inequalities)  Suppose that $R$ is a Noetherian $d$-dimensional  local ring with $\dim N(\hat R)<d$, $M$ is a finitely generated $R$-module and $\mathcal I(1)=\{I(1)_j\}$ and $\mathcal I(2)=\{I(2)_j\}$ are filtrations of $R$ by $m_R$-primary ideals. Then 
\begin{enumerate}
\item[1)] For $1\le i\le d-1$,
$$
\begin{array}{l}
e_R(\mathcal I(1)^{[i]},\mathcal I(2)^{[d-i]};M)^2\\
\le e_R(\mathcal I(1)^{[i+1]},\mathcal I(2)^{[d-i-1]};M)e_R(\mathcal I(1)^{[i-1]},\mathcal I(2)^{[d-i+1]};M)\end{array}
$$
\item[2)]  For $0\le i\le d$, 
$$
e_R(\mathcal I(1)^{[i]},\mathcal I(2)^{[d-i]};M)e_R(\mathcal I(1)^{[d-i]},\mathcal I(2)^{[i]};M)\le e_R(\mathcal I(1);M)e_R(\mathcal I(2);M),
$$
\item[3)] For $0\le i\le d$, $e_R(\mathcal I(1)^{[d-i]},\mathcal I(2)^{[i]};M)^d\le e_R(\mathcal I(1);M)^{d-i}e_R(\mathcal I(2);M)^i$ and
\item[4)]  $e_R(\mathcal I(1)\mathcal I(2));M)^{\frac{1}{d}}\le e_R(\mathcal I(1);M)^{\frac{1}{d}}+e_R(\mathcal I(2);M)^{\frac{1}{d}}$, 

where $\mathcal I(1)\mathcal I(2)=\{I(1)_jI(2)_j\}$.
\end{enumerate}
\end{Theorem}

The Minkowski inequalities were formulated and proven for  $m_R$-primary ideals by Teissier \cite{T1}, \cite{T2} and proven in full generality, for $m_R$-primary ideals in Noetherian local rings,  by Rees and Sharp \cite{RS}. The fourth inequality, 
was proven for filtrations  of $R$ by $m_R$-primary ideals in a regular local ring with algebraically closed residue field by Musta\c{t}\u{a} (\cite[Corollary 1.9]{Mus}) and more recently by Kaveh and Khovanskii (\cite[Corollary 7.14]{KK1}). The inequality 4) was proven with our assumption that $\dim N(\hat R)<d$ in \cite[Theorem 3.1]{C3}.

Another important property is that the mixed multiplicities are always nonnegative (\cite[Proposition 1.3]{CSV}).
In contrast, the mixed multiplicities are strictly positive in the classical case of $m_R$-primary ideals (\cite{T1} or \cite[Corollary 17.4.7]{HS}).

Suppose that $R$ is a $d$-dimensional excellent local domain, with quotient field $K$. A valuation $\nu$ of $K$ is called an $m_R$-valuation if $\nu$ dominates $R$ ($R\subset V_{\nu}$ and $m_{\nu}\cap R=m_R$ where $V_{\nu}$ is the valuation ring of $\nu$ with maximal ideal $m_{\nu}$) and ${\rm trdeg}_{R/m_R}V_{\nu}/m_{\nu}=d-1$.

Suppose that $I$ is an ideal in $R$. Let $X$ be the normalization of the blowup of $I$, with projective birational morphism $\phi:X\rightarrow \mbox{Spec}(R)$. Let $E_1,\ldots,E_t$ be the irreducible components of $\phi^{-1}(V(I))$ (which necessarily have dimension $d-1$). The Rees valuations of $I$ are the discrete valuations $\nu_i$ for $1\le i\le t$ with valuation rings $V_{\nu_i}=\mathcal O_{X,E_i}$. If $R$ is normal, then $X$ is equal to the blowup of the integral closure $\overline{I^s}$ of an appropriate power $I^s$ of $I$.

Every Rees valuation $\nu$ that dominates $R$  is an $m_R$-valuation and every $m_R$-valuation is a Rees valuation of an $m_R$-primary ideal by \cite[Statement  (G)]{R3}.

Associated to an $m_R$-valuation $\nu$ are valuation ideals
\begin{equation}\label{eqX2}
I(\nu)_n=\{f\in R\mid \nu(f)\ge n\}
\end{equation}
for $n\in \NN$.
In general, the filtration $\mathcal I(\nu)=\{I(\nu)_n\}$ is not Noetherian. 
In  a two-dimensional normal local ring $R$, the condition that the  filtration of valuation ideals of $R$ is Noetherian for all  $m_R$-valuations dominating $R$ is the condition (N) of Muhly and Sakuma \cite{MS}. It is proven in \cite{C8} that a complete normal local ring of dimension two satisfies condition (N) if and only if its divisor class group is a torsion group. 
An example is given in \cite{CGP} of an $m_R$-valuation of a 3-dimensional regular local ring $R$ such that the filtration is not Noetherian. 

\begin{Definition}\label{DefDF} Suppose that $R$ is an excellent local domain. We
say that a filtration $\mathcal I$ of $R$ by $m_R$-primary ideals is a divisorial filtration if there exists a projective birational morphism $\phi:X\rightarrow \mbox{Spec}(R)$ such that $X$ is the normalization of the blowup of an $m_R$-primary ideal and there exists a nonzero effective Cartier divisor $D$ on $X$ with exceptional support for $\phi$ such that 
$\mathcal I=\{I(mD)\}_{m\in\NN}$ where
\begin{equation}\label{eqX22}
I(mD)=\Gamma(X,\mathcal O_X(-mD))\cap R .
\end{equation}
\end{Definition}
We will write 
$$
\mathcal I(D)=\{I(mD)\}_{m\in \NN}.
$$
If $R$ is normal, then $I(mD)=\Gamma(X,\mathcal O_X(-mD))$.
If $D=\sum_{i=1}^ta_iE_i$ where the $a_i\in \NN$ and the $E_i$ are prime exceptional divisors of $\phi$, with associated $m_R$-valuations $\nu_i$, then 
$$
I(mD)=I(\nu_1)_{a_1m}\cap\cdots\cap I(\nu_t)_{a_tm}.
$$

Rees has shown in \cite{R} that if $R$ is a formally equidimensional Noetherian local ring and $I\subset I'$ are $m_R$-primary ideals such that $e_R(I;R)=e_R(I';R)$, then $\bigoplus_{n\ge 0}(I')^n$ is integral over $\bigoplus_{n\ge 0}I^n$ ($I$ and $I'$ have the same integral closure). An exposition of this converse to the above cited \cite[Proposition 11.2.1]{HS} is given in   \cite[Proposition 11.3.1]{HS}, in the section entitled ``Rees's Theorem''. Rees's theorem is not true in general for filtrations of $m_R$-primary ideals (a simple example in a regular local ring is given in \cite{CSS}) but it is true for divisorial filtrations.

In Theorem \cite[Theorem 3.5]{C5}, it is shown that Rees's theorem is true for divisorial filtrations of an excellent local domain.

\begin{Theorem}(\cite[Theorem 1.4 and Theorem 3.5]{C5})  
Suppose that $R$ is a $d$-dimensional excellent local domain. Let $\phi:X\rightarrow \mbox{Spec}(R)$ be the normalization of the blowup of an $m_R$-primary ideal. Suppose that $D_1$ and $D_2$ are effective Cartier divisors on $X$ with exceptional support such that $D_1\le D_2$. Then
 $$
 e_R(\mathcal I(D_1);R)=e_R(\mathcal I(D_2);R)
 $$
  if and only if 
$$
I(mD_1)=I(mD_2)\mbox{ for all }m\in \NN.
$$

\end{Theorem}

An algebraic local ring is a local domain which is essentially of finite type over a field. 
Let $R$ be a $d$-dimensional normal algebraic local ring and $\phi:X\rightarrow \mbox{Spec}(R)$ be the  blow up of an $m_R$-primary ideal of $R$ such that $X$ is normal. 

In \cite{C5}, anti-positive intersection products $\langle(-D_1)^{d_1}\cdot\ldots \cdot (-D_r)^{d_r}\rangle$
where $D_1,\ldots,D_r$ are effective Cartier divisors on $X$ with exceptional support
 are defined, generalizing the positive intersection product of Cartier divisors defined on projective varieties in \cite{BFJ} over an algebraically closed field of characteristic zero and in \cite{C4} over an arbitrary field. The anti-positive intersection multiplicities have the property that if $d_1+\cdots+d_r=d$, then $\langle (-D_1)^{d_1}\cdot,\ldots,\cdot (-D_r)^{d_r}\rangle$ is a non positive real number.

 It is shown in \cite[Theorem 8.3]{C5} that  we have  identities
$$
e_R(\mathcal I(D_1)^{[d_1]},\ldots,\mathcal I(D_r)^{[d_r]};R)=-\langle(-D_1)^{d_1}\cdot,\ldots,\cdot (-D_r)^{d_r}\rangle.
$$
In particular, $e_R(\mathcal I(D);R)=-\langle(-D)^d\rangle$.  Thus by (\ref{eqV6}), we have that 
 
 \begin{equation}\label{eq10}
 \begin{array}{l}
\lim_{m\rightarrow \infty}  \frac{\ell_R(R/I(mn_1D_1)\cdots\mathcal I(mn_rD_r))}{m^d}\\
  =-\sum_{d_1+\cdots +d_r=d}\frac{1}{d_1!\cdots d_r!}\langle(D_1)^{d_1}\cdot\ldots\cdot (-D_r)^{d_r}\rangle
n_1^{d_1}\cdots n_r^{d_r}.
\end{array}
\end{equation} 
 From the case $r=1$, we obtain
 \begin{equation}\label{eq11}
 \lim_{m\rightarrow \infty}\frac{\ell_R(R/I(mD))}{m^d}=-\frac{\langle(-D)^d\rangle}{d!}.
 \end{equation}
 The interpretation of mixed multiplicities as anti-positive intersection multiplicities is particularly useful in the calculation of examples. This is the method we use in the example constructed in this paper.

The multiplicity of a ring with respect to a non Noetherian filtration can be an irrational number. 
The following is a very simple example of a filtration of $m_R$-primary ideals such that the multiplicity is not rational. Let $k$ be a field and $R=k[[x]]$ be a power series ring over $k$. Let $I_n =(x^{\lceil n\sqrt{2}\rceil})$   
where $\lceil \alpha\rceil$ is the round up of a real number $\alpha$ (the smallest integer which is greater than or equal to $\alpha$).  Then $\mathcal I=\{I_n\}$ is a filtration of $m_R$-primary ideals such that 
$$
e_R(\mathcal I;R)=\lim_{n\rightarrow\infty}\frac{\ell_R(R/I_n)}{n} =\sqrt{2}
$$ 
is an irrational number. 

There are also irrational examples determined by the valuation ideals of a  discrete valuation. In Example 6 of \cite{CS} an example is given of a normal 3 dimensional  local ring $R$ which is essentially of finite type over a field of arbitrary characteristic and a divisorial valuation $\nu$ on the quotient field of $R$ which dominates $R$ such that the filtration of $m_R$-primary ideals $\mathcal I(\nu)=\{I(\nu)_n\}$
defined by 
$$
I(\nu)_n=\{f\in R\mid \nu(f)\ge n\}
$$
satisfies that the limit 
$$
e_R(\mathcal I(\nu);R)=lim_{n\rightarrow\infty}\frac{\ell_R(R/I_n)}{n^3/3!} 
$$ 
is irrational.  In this paper we give an example of this behavior in (\ref{eq30}), and give examples of irrational mixed multiplicities.

We define a multigraded filtration $\mathcal I=\{I_{n_1,\ldots,n_r}\}_{n_1,\ldots,n_r\in\NN}$ of ideals on a ring $R$ to be a collection of ideals of $R$ such that
 $R=I_{0,\ldots,0}$, 
 $$
 I_{n_1,\ldots,n_{n_{j-1}},n_j+1,n_{j+1},\ldots,n_r}\subset I_{n_1,\ldots,n_{j-1},n_j,n_{j+1},\ldots,n_r}
 $$
  for all $n_1,\ldots,n_r\in\NN$ and
 $$
 I_{a_1,\ldots,a_r}I_{b_1,\ldots,b_r}\subset I_{a_1+b_1,\ldots,a_r+b_r}
 $$
  whenever $a_1,\ldots,a_r,b_1,\ldots,b_r\in \NN$.
 
 A multigraded filtration
 $\mathcal I=\{I_{n_1,\ldots,n_r}\}$ of ideals on a local ring $R$ is a  multigraded filtration of $R$ by $m_R$-primary ideals if $I_{n_1,\ldots,n_r}$ is $m_R$-primary whenever $n_1+\cdots+n_r>0$.

 In \cite[Section 7]{CSS}, we  give an example showing that  the mixed multiplicities of filtrations (explained between (\ref{M2}) and (\ref{eqV6}) of this paper) do not have a good extension to arbitrary multigraded non Noetherian filtrations $\mathcal I=\{I_{n_1,\ldots,n_r}\}$ of $m_R$-primary ideals, even in a power series ring in one variable over a field. In the  example in \cite{CSS}, we have $d=1$ and the function
 $$
 P(n_1,n_2)=\lim_{m\rightarrow \infty}\frac{\ell_R(R/I_{mn_1,mn_2})}{m}=\lceil \sqrt{n_1^2+n_2^2}\rceil
$$ 
 for $n_1,n_2\in \NN$, where $\lceil x\rceil$ is the round up of a real number $x$. The function $P(n_1,n_2)$ is  far from polynomial like. 
 
 However, it is shown  in \cite[Section 7]{CSS} that the function $P(n_1,\ldots,n_r)$ is polynomial like in an important situation. 
 Let $R$  be an excellent, normal, two dimensional local ring, and $\phi:X\rightarrow \mbox{Spec}(R)$ be a resolution of singularities. Let $E_1,\ldots,E_r$ be the irreducible exceptional divisors of $\phi$, and let
 $\{I_{n_1,\ldots,n_r}\}$ be the
   multigraded filtration  of $m_R$-primary ideals defined by
   $$
   I_{n_1,\ldots,n_r}=I(n_1E_1+\cdots +n_rE_r).
   $$
   It is shown in \cite[Section 7]{CSS} (using some results from \cite{CHR}) that
  the function
 \begin{equation}\label{eq21}
 P(n_1,\ldots,n_r)=\lim_{m\rightarrow \infty}\frac{\ell_R(R/I_{mn_1,\ldots,mn_r})}{m^2}
\end{equation} 
is  a piecewise rational polynomial  function on an  abstract complex of rational polyhedral sets whose union is $(\QQ_{\ge 0})^r$. This holds, even though the  filtration $\{I_{n_1,\ldots,n_r}\}$ is generally not Noetherian.

The function $P(n_1,\ldots,n_r)$ of (\ref{eq21}) is in fact  given by the anti-intersection product (\ref{eq11}) on the resolution of singularities. The anti-positive intersection product of (\ref{eq11}) is in this case the ordinary intersection product of the Zariski decomposition of $-D$ (where $D=n_1E_1+\cdots+n_rE_r$).

In Theorem  \ref{Theorem2} below, we give an example of such a function on a resolution of singularities of a normal three dimensional excellent local ring with two irreducible exceptional divisors such that the function $P(n_1,n_2)$
is a piecewise polynomial function on an abstract complex of polyhedral sets, but the polynomials and the polyhedral sets are not rational.

 In this paper, we compute  the functions 
 $$
 \lim_{m\rightarrow\infty}
  \frac{\ell_R(R/I(mD))}{m^3}\mbox{  and }\lim_{m\rightarrow\infty}\frac{\ell_R(R/I(mn_1D_1)I(mn_2D_2))}{m^d}
  $$  
  and the associated multiplicities and mixed multiplicities 
  $$
  e_R(\mathcal I(D);R)\mbox{ and }e_R(\mathcal I(D_1),\mathcal I(D_2);R)
  $$
in a  specific example.  We compute the anti-positive intersection numbers in order to determine these functions. Let $k$ be an algebraically closed field. We construct a 3-dimensional normal algebraic local ring $R$ over $k$ and the blow up $\phi:X\rightarrow \mbox{Spec}(R)$  of an $m_R$-primary ideal such that $X$ is nonsingular with two  irreducible exceptional divisors, which we denote by $\overline S$ and $F$. Theorems \ref{Theorem2} and \ref{Theorem3} below refer to this example. These theorems are proven in this paper.  

 The resolution of singularities of a three dimensional normal local ring which we construct is similar to    the one constructed in  \cite[Example 6]{CS}, referred to above, which is used to give an example of a  valuative filtration  with irrational multiplicity.  In \cite[Example 6]{CS}, no details of the construction or analysis of the example are given.  We give complete details in this paper. We illustrate the application of anti-positive intersection products in the proof.

  \begin{Theorem}\label{Theorem2} Let $D=n\overline S+jF$ with $n,j\in \NN$.
   Then 
    $$
   \lim_{m\rightarrow\infty}
  \frac{\ell_R(R/I(mD))}{m^3} 
  =\left\{\begin{array}{ll}
  33 n^3&\mbox{ if }j<n\\
  78n^3-81n^2j+27nj^2+9j^3&\mbox{ if }n\le j<n\left(3-\frac{\sqrt{3}}{3}\right)\\
  \left(\frac{2007}{169}-\frac{9\sqrt{3}}{338}\right)j^3&\mbox{ if }n\left(3-\frac{\sqrt{3}}{3}\right)<j.
  \end{array}\right.
  $$
    \end{Theorem}
  
  In particular, 
  $$
  e_R(\mathcal I(D);R)
  =\left\{\begin{array}{ll}
  198 n^3&\mbox{ if }j<n\\
  468n^3-486n^2j+162nj^2+54j^3&\mbox{ if }n\le j<n\left(3-\frac{\sqrt{3}}{3}\right)\\
  \left(\frac{12042}{169}-\frac{27\sqrt{3}}{169}\right)j^3&\mbox{ if }n\left(3-\frac{\sqrt{3}}{3}\right)<j.
  \end{array}\right.
  $$

  As a consequence, we have that 
  \begin{equation}\label{eq20}
  \lim_{m\rightarrow \infty}\frac{\ell_R(R/I(mF))}{m^3}=\left(\frac{2007}{169}-\frac{9\sqrt{3}}{338}\right)
  \end{equation}
  and 
  \begin{equation}\label{eq30}
  e_R(\mathcal I(F);R)=\frac{12042}{169}-\frac{27\sqrt{3}}{169},
  \end{equation}
 giving an example of a divisorial valuation $\nu=\nu_F$ dominating $R$ such that 
  $e_R(\mathcal I(\nu);R)=e_R(\mathcal I(F);R)$ is an irrational number, where 
  $\mathcal I(\nu)=\{I(\nu)_m\}$ and $I(\nu)_n$ is the valuation ideal  $I(\nu)_m=\{f\in R\mid \nu(f)\ge m\}$.

  \begin{Theorem}\label{Theorem3} For $n,j\in \NN$,
  $$
  \begin{array}{l}
 \lim_{m\rightarrow\infty}\frac{\ell_R(R/I(mn\overline S)I(mjF))}{m^3}\\
 = 33n^3+\left(\frac{891}{26}+\frac{99\sqrt{3}}{26}\right)n^2j+\left(\frac{6021}{169}-\frac{27\sqrt{3}}{338}\right)nj^2
 +\left(\frac{2007}{169}-\frac{9\sqrt{3}}{338}\right)j^3.
 \end{array}
 $$
 \end{Theorem}
 
In particular, the mixed multiplicities are
$$
\begin{array}{l}
e_R(\mathcal I(\overline S)^{[3]};R)=e_R(\mathcal I(\overline S);R)=198\\
e_R(\mathcal I(\overline S)^{[2]},\mathcal I(F)^{[1]};R)=\frac{891}{13}+\frac{99\sqrt{3}}{13}\\
e_R(\mathcal I(\overline S)^{[1]},\mathcal I(F)^{[2]};R)=\frac{12042}{169}-\frac{27\sqrt{3}}{169}\\
e_R(\mathcal I(F)^{[3]};R)=e_R(\mathcal I(F);R)=\frac{12042}{169}-\frac{27\sqrt{3}}{169}.\end{array}
$$

\section{Anti-positive intersection products}\label{SecPIP}

\subsection{The construction of anti-positive intersection products} In this subsection we review the construction of anti-positive intersection products in \cite{C5}. 

 Anti-positive intersection products generalize the positive intersection products of Cartier divisors defined on projective varieties in \cite{BFJ} over an algebraically closed field of characteristic zero and in \cite{C4} over an arbitrary field.

Let $K$ be an algebraic function field over a field $k$. An algebraic local ring of $K$ is a local ring $R$ that is a localization of a finitely generated $k$-algebra and is a domain whose quotient field is $K$ with maximal ideal $m_R$. Let $R$ be a $d$-dimensional algebraic normal local ring of $K$. Let ${\rm BirMod}(R)$ be the directed set of blowups 
$\phi:X\rightarrow \mbox{Spec}(R)$  of an $m_R$-primary ideal $I$ of $R$ such that $X$ is normal.

Suppose that $\phi:X\rightarrow \mbox{Spec}(R)$ is in ${\rm BirMod}(R)$.
Let $\{E_1,\ldots,E_t\}$ be the irreducible exceptional divisors of $\phi$.  We define $M^1(X)$ to be the subspace of the real  vector space $E_1\RR+\cdots + E_t\RR$ that is generated by the Cartier divisors. An element of $M^1(X)$ will be called an $\RR$-divisor on $X$. We will say that $D\in M^1(X)$ is a $\QQ$-Cartier divisor if there exists $n\in \ZZ_+$ such that $nD$ is a Cartier divisor. 

In Section 6.1 of \cite{C5}, we define a natural intersection product $(D_1\cdot D_2\cdot\ldots\cdot D_d)$ on $X$ for $D_1,\ldots,D_d\in M^1(X)$.  The intersection product is a restriction of the one defined in \cite{Kl}. 

We will say that a divisor $F=a_1E_1+\cdots+a_tE_t\in M^1(X)$ is effective if $a_i\ge 0$ for all $i$, and anti-effective if $a_i\le 0$ for all $i$. This defines a partial order $\le$ on $M^1(X)$ by $A\le B$ if $B-A$ is effective.  The effective cone ${\rm EF}(X)$ is the closed convex cone in $M^1(X)$  of effective $\RR$-divisors.
The anti-effective cone ${\rm AEF}(X)$ is the closed convex cone in $M^1(X)$ consisting of all anti-effective $\RR$-divisors. 

We will say that an anti-effective divisor $F\in M^1(X)$ is  numerically effective (nef) if 
$$
(F\cdot C) \ge 0
$$
 for all closed curves $C$ in $\phi^{-1}(m_R)$.  The nef cone $\mbox{Nef}(X)$ is the closed convex cone in $M^1(X)$ of all nef $\RR$-divisors on $X$. 
There is an  inclusion of cones ${\rm Nef}(X)\subset {\rm AEF}(X)$.
 We define a divisor $F\in M^1(X)$ to be  ample if $F$ is a formal sum $F=\sum a_iF_i$ where $F_i$ are ample anti-effective Cartier divisors and   $a_i$ are positive real numbers. A divisor $D$ is anti-ample if $-D$ is ample. We define the convex cone
$$
\mbox{Amp}(X)=\{F\in M^1(X)\mid F\mbox{ is ample}\}.
$$

We have that $\mbox{Amp}(X)\subset\mbox{Nef}(X)$, the closure of $\mbox{Amp}(X)$ is $\mbox{Nef}(X)$,
and the interior of $\mbox{Nef}(X)$ is $\mbox{Amp}(X)$, as in \cite{Kl}, \cite[Theorem 1.4.23]{LV1}.

Suppose that $X\in{\rm BirMod}(R)$. Let $E_1,\ldots,E_r$ be the exceptional components of $X$ for the morphism $X\rightarrow \mbox{Spec}(R)$.
 For $0< p\le d$, we define
 $M^p(X)$ to be the direct product of $M^1(X)$ $p$ times, and we define $M^0(X)=\RR$.
 For $1< p\le d$,  we define
 $L^p(X)$ to be the vector space of $p$-multilinear forms from $M^p(X)$ to $\RR$, and define $L^0(X)=\RR$.

The intersection product gives us 
 $p$-multilinear maps
\begin{equation}\label{eq4*}
M^p(X)\rightarrow L^{d-p}(X)
\end{equation}
 for $0\le p\le d$.
 
We have that ${\rm BirMod}(R)$ is a directed set by the $R$-morphisms $Y\rightarrow X$ for $X,Y\in {\rm BirMod}(R)$.     There is at most one $R$-morphism $X\rightarrow Y$ for $X,Y\in {\rm BirMod}(X)$.

The set 
$\{M^p(Y_i)\mid Y_i\in {\rm BirMod}(R)\}$ is a directed system of real vector spaces, where we have a linear mapping 
$f_{ij}^*:M^p(Y_i)\rightarrow M^p(Y_j)$ if the natural birational map $f_{ij}:Y_j\rightarrow Y_i$ is an $R$-morphism. 
We define 
$$
M^p(R)=\lim_{\rightarrow}M^p(Y_i)
$$

Anti-positive intersection products $\langle\alpha_1\cdot \ldots \cdot \alpha_p\rangle$ for anti-effective $\alpha_1,\ldots,\alpha_p\in M^1(R)$ are defined  in \cite[Definition 7.4]{C5}, generalizing the positive intersection products defined 
on projective varieties  in \cite{BFJ} over an algebraically closed field of characteristic zero, and
in \cite[Definition 4.4]{C4} over an arbitrary field. The anti-positive intersection product $\langle \alpha_1\cdot\ldots\cdot \alpha_d\rangle$ of $d$ anti-effective  divisors $\alpha_1,\ldots,\alpha_d\in M^1(R)$ is always a non positive real number.

   The proof of the following proposition is similar to that of \cite[Proposition 4.12]{C4}, replacing the reference to \cite[Proposition 4.3]{C4} with \cite[Proposition 7.3]{C5}, and replacing the use of the continuity statement of \cite[Proposition 4.7]{C4}  with the continuity statement of \cite[Proposition 7.5]{C5}.
   
     \begin{Proposition}\label{Prop2}
     Suppose that $\alpha_1,\ldots,\alpha_p\in M^1(R)$ are anti-effective. Then the anti-positive intersection product $\langle\alpha_1\cdot\ldots\cdot\alpha_p\rangle$ is the least upper bound  of all ordinary intersection products 
     $\beta_1\cdot\ldots\cdot\beta_p$ in $L^{d-p}(R)$ with $\beta_i\in M^1(R)$ anti-effective and nef and $\beta_i\le \alpha_i$.
\end{Proposition}

\subsection{$\gamma_{E}(D)$ and anti-positive intersection products}

Let $\phi:X\rightarrow \mbox{spec}(R)\in\mbox{BirMod}(R)$. 
   Let $E_1,\ldots,E_t$ be the irreducible exceptional divisors of $\phi$.

 Suppose that $D=\sum a_iE_i$ is an effective $\QQ$-Cartier divisor on $X$. If $D$ is Cartier, then $\Gamma(X,\mathcal O_X(-D))$ is an $m_R$-primary ideal since $R$ is normal. Write $I(D)=\Gamma(X,\mathcal O_X(-D))$.
   
   Let $\nu_{E_i}$  be the natural discrete valuation with valuation ring $\mathcal O_{X,E_i}$.

   Let $r$ be a fixed positive integer such that $rD$ is a Cartier divisor. 
      Define
   $$
   \tau_{rm,E_i}(rD)=\min\{\nu_{E_i}(f)\mid f\in \Gamma(X,\mathcal O_X(-mrD))\},
   $$
   and $\gamma_{E_i}(D)=\inf _m\frac{\tau_{rm,E_i}(rD)}{rm}$. The real number $\gamma_{E_i}(D)$ is independent of $r$.
      We have that    
      $$
      \gamma_{E_i}(D)\ge a_i
      $$
       for all $i$, and 
   $$
  I(mrD)= \Gamma(X,\mathcal O_X(-mrD))=\Gamma(X,\mathcal O_X(-\lceil{\sum mr\gamma_{E_i}(D)E_i}\rceil)
   $$
   for all $m\in \NN$ (this is shown in \cite[Lemma 3.1]{C5}). Here $\lceil x\rceil$ denotes the round up of a real number $x$.

   \begin{Lemma}\label{Lemma3} Suppose that $D$ is anti-nef. Then $\gamma_{E_i}(D)=a_i$ for all $i$.
   \end{Lemma}
   
   \begin{proof}
   If $-D$ is ample, then $\mathcal O_X(-mrD)$ is generated by global sections if $r$ is such that $rD$ is Cartier and $m\gg 0$. Thus for all $i$,  there exists $f\in\Gamma(X,\mathcal O_X(-rmD))$ such that $\nu_{E_i}(f)=mra_i$. Thus $\gamma_{E_i}(D)=a_i$.
   
   Now suppose that  $-D$ is nef. Given $\epsilon>0$, there exists an anti-ample effective $\QQ$-Cartier divisor $A$ on $X$ such that $A=\sum c_iE_i$ with $a_i\le c_i<a_i+\epsilon$ for all $i$. Let $r$ be such that $rA$ and $rD$ are Cartier divisors. For $m\gg 0$, there exists $f\in \Gamma(X,\mathcal O_X(-mrA))$ such that $\nu_{E_i}(f)=mrc_i<mra_i+mr\epsilon$. Thus $\gamma_{E_i}(D)\le a_i+\epsilon$. Since this is true for all $\epsilon$, we have that $\gamma_{E_i}(D)=a_i$.
   \end{proof}

  \begin{Lemma}\label{Lemma1} Suppose that $D\in M^1(X)$ is an effective $\QQ$-Cartier divisor such that $-\sum\gamma_{E_i}(D)E_i$ is nef. Suppose that $Y\rightarrow \mbox{Spec}(R)\in\mbox{BirMod}(R)$ and there exists a factorization $\psi:Y\rightarrow X$.
  If $G\in M^1(Y)$ is an effective and
  anti-nef $\QQ$-Cartier divisor such that  $-G\le \psi^*(-D)$ then $-G\le   \psi^*(-\sum \gamma_{E_i}(D)E_i)$.
  \end{Lemma}

  \begin{proof} Let $F_1,\ldots,F_s$ be the irreducible exceptional divisors of $Y\rightarrow \mbox{Spec}(R)$.
    Since $-\sum\gamma_{E_i}(D)E_i$ is nef, we have that 
  $\sum \gamma_{F_i}(\psi^*(D))F_i=\psi^*(\sum \gamma_{E_i}(D)E_i)$.  Write $G=\sum g_iF_i$. Since $-G$ is nef, we have that $\gamma_{F_i}(G)=g_i$ for all $i$. Since
  $$
  \mathcal O_Y(-mrG)\subset \mathcal O_Y(-\psi^*(mrD))
  $$
   whenever $rG$, $rD$ are Cartier divisors and $m>0$, we have that $\gamma_{F_i}(D)\le \gamma_{F_i}(G)=g_i$ for all $i$. Thus
  $$
  -G\le -\sum\gamma_{F_i}(\psi^*(D))F_i=-\psi^*(\sum \gamma_{E_i}(D)E_i).
  $$
 \end{proof}
 
 The following proposition is a consequence of Proposition \ref{Prop2} and \ref{Lemma1}.
 
 \begin{Proposition}\label{Prop1} Suppose that $D_1,\ldots,D_d\in M^1(X)$ are effective $\QQ$-Cartier divisors such that the divisors $-\sum\gamma_{E_i}(D_j)E_i$ are nef for $1\le j\le d$.
  Then the positive intersection product $\langle-D_1\cdot,\ldots,\cdot -D_d\rangle$ is the ordinary intersection product
  $(-\sum\gamma_{E_i}(D_1)E_i\cdot\ldots\cdot-\sum\gamma_{E_i}(D_d)E_i)$. 
  \end{Proposition}

\section{Intersection theory on projective varieties} 
In this section we review some material on intersection theory on Projective varieties. We refer to \cite{Kl} and \cite{LV1}. 
Let $k$ be an algebraically closed field. Let $T$ be a nonsingular projective surface over $k$. Then 
$(\mbox{Pic}(T)/\equiv)\otimes\RR$, where $\equiv$ denotes numerical equivalence, is a finite dimensional real vector space. We will often abuse notation, identifying the class of an invertible sheaf $\mathcal O_T(D)$ with the class of the divisor $D$.

We will denote  the closure  of the real cone in $(\mbox{Pic}(T)/\equiv)\otimes\RR$ generated by the classes of effective divisors by  $\overline{\mbox{Eff}}(T)$ and   the closure  of the real cone in $(\mbox{Pic}(T)/\equiv)\otimes\RR$ generated by the classes of numerically effective divisors by $\overline{\mbox{Nef}}(T)$, and the closure of the real cone   in $(\mbox{Pic}(T)/\equiv)\otimes\RR$ generated by the classes of ample divisors by $\overline{\mbox{Amp}}(T)$.

If $V$ is a nonsingular $r$-dimensional projective variety and $D_1,\ldots,D_r$ are divisors on $V$ we will denote the intersection product of $D_1,\ldots,D_r$ on $V$ by $(D_1\cdot D_2\cdot\ldots\cdot D_r)_V$. When there is no danger of confusion about the ambient variety, we will simply write $(D_1\cdot D_2\cdot\ldots\cdot D_r)$.

\section{An example} In this section, we construct a resolution of singularities of a  three dimensional normal local ring, and compute the multiplicities and mixed multiplicities of its divisorial filtrations.   This resolution of singularities  is similar to    the one constructed in  \cite[Example 6]{CS}, which is used to give an example of a   filtration  of a divisorial valuation with irrational multiplicity.  In \cite[Example 6]{CS}, no details of the construction or analysis of the example are given.  We give complete details in this section.

Let $k$ be an algebraically closed field, Let
$W$ be an elliptic curve over $k$ and  $S$ be the abelian surface $S=W\times W$. Let $\pi_1:S\rightarrow W$ and $\pi_2:S\rightarrow W$ be the two projections. Let $p\in W$ be  a closed point, $A=\pi_1^{*}(p)$, $B=\pi_2^{*}(p)$ and $\Delta\subset W\times W=S$ be the diagonal. Let $V$ be the real subspace of $(\mbox{Pic}(S)/\equiv)\otimes\RR$  generated by the classes of $A,B$ and $\Delta$.   As shown in \cite{C7} and \cite[Example 4]{CS}, we have that $V$ has dimension 3 and 
$$
(\Delta^2)=(A^2)=(B^2)=0\mbox{ and }(A\cdot B)=(A\cdot\Delta)=(B\cdot \Delta)=1.
$$
Further, 
$\overline{\mbox{Amp}}(S)= \overline{\mbox{Eff}}(S)=\overline{\mbox{Nef}}(S)$, and $V\cap \overline{\mbox{Eff}}(S)$
is the real cone which is the component of 
$$
\{xA+yB+z\Delta\mid (xA+yB+z\Delta)^2\ge 0\}
$$
which contains the ample divisor $A+B+\Delta$.  

If $j,n\ge 0$ we have that 
\begin{equation}\label{eq1}
n(A+2B+ 3\Delta)-j(A+B+\Delta)\in \overline{\mbox{Eff}}(S)\mbox{ if and only if }j<n\left(2-\frac{\sqrt{3}}{3}\right).
\end{equation}
The canonical divisor of the Abelian surface $S$ is $K_S=0$.

Let $X$ be the projective bundle $X=\PP(\mathcal O_S(-3(A+2B+3\Delta))\oplus\mathcal O_S)$ with projection $\pi:X\rightarrow S$. Identify  the section of $\pi$ corresponding to the natural surjection of $\mathcal O_S$-modules
$$
\mathcal O_S(-3(A+2B+3\Delta))\oplus\mathcal O_S\rightarrow \mathcal O_S(-3(A+2B+3\Delta))
$$
with $S$ (c.f. \cite[Proposition 7.12]{Ha}).
Then $\mathcal O_X(1)\cong \mathcal O_X(S)$ and 
$$
\mathcal O_X(S)\otimes\mathcal O_S\cong \mathcal O_S(-3(A+2B+3\Delta)).
$$

A canonical divisor on $X$ is $K_X=-2S+\pi^*(-3(A+2B+3\Delta))$ (this can be seen by applying adjunction on a fiber of $\pi$ and then on the section $S$). The Picard group of $X$ is 
$$
\mbox{Pic}(X)=\mathcal O_X(1)\ZZ\oplus \pi^*\mbox{Pic}(S).
$$
Suppose that $\Gamma$ is an effective divisor on $X$. Then $\Gamma\sim nS+\pi^*(L)$ for some divisor $L$ on $X$. We have that  $n=(\Gamma\cdot g)\ge 0$ for a  fiber $g$ of $\pi$. Since $\Gamma$ is effective,
$$
\begin{array}{lll}
0<h^0(X,\mathcal O_X(\Gamma))&=&h^0(S,\mbox{Sym}^n\left(\mathcal O_S(-3(A+2B+3\Delta))\oplus\mathcal O_S\right)\otimes\mathcal O_S(L)))\\
&=&\sum_{i=0}^nh^0(S,\mathcal O_S(L-i3(A+2B+3\Delta))).
\end{array}
$$
Thus $L\in\overline{\mbox{Eff}}(S)$.

Let $T$ be the section of $\pi$ corresponding to the surjection of $\mathcal O_S$-modules 
$$
\mathcal O_S(-3(A+2B+3\Delta))\oplus \mathcal O_S\rightarrow \mathcal O_S.
$$
 Then $\mathcal O_X(1)\otimes\mathcal O_T\cong\mathcal O_S$ so that $T\cap S=\emptyset$. Further, 
$\mathcal O_X(T)\cong \mathcal O_X(S+3(A+2B+3\Delta))$. Now $3(A+2B+3\Delta)$ is very ample on $S$ (by the theorem in Section 17 \cite{Mum}), so the complete linear system $|T|$ on $X$ is base point free, and thus induces a projective morphism $X\rightarrow Z$ which contracts $S$. Suppose that $\gamma$ is an irreducible curve on $X$ which is not contained in $S$ and is not a fiber of $\pi$. Let $\overline \gamma$ be the image of $\gamma$ by $\pi$ in $S$ which is a curve. We have that 
$$
\begin{array}{lll}
(\gamma\cdot T)_X&=&(\gamma\cdot (S+\pi^*(3(A+2B+3\Delta))))_X\\
&=&\deg(\mathcal O_X(S+\pi^*(3(A+2B+3\Delta)))\otimes \mathcal O_{\gamma})\\
&\ge& \deg (\mathcal O_X(\pi^*(3(A+2B+3\Delta)))\otimes\mathcal O_{\gamma})\\
&=&\deg(\gamma/\overline\gamma)(\overline\gamma\cdot 3(A+2B+3\Delta))_S>0
\end{array}
$$
by the projection formula (\cite[Example 7.1.9]{F}), and since $A+2B+3\Delta$ is ample.
Thus $X\setminus S\rightarrow Z$ is finite to one. Let $\overline Z$ be the normalization of $Z$ in the function field of $X$. Then there is an induced birational projective morphism $\lambda:X\rightarrow \overline Z$ such that $S$ is contracted to a point $q$ of $\overline Z$ and $X\setminus S\rightarrow \overline Z-q$ is an isomorphism.
The divisor $-S$ is relatively ample for $\lambda$ since $\mathcal O_X(-S)\otimes \mathcal O_S$ is ample on $S$. Thus there exists $n>0$ such that $X$ is the blow up of the ideal sheaf $\lambda_*\mathcal O_X(-nS)$ of $\overline Z$, which has the property that the support of $\mathcal O_{\overline Z}/\lambda_*\mathcal O_X(-nS)$ is the point $q$.

The divisor $A+B+\Delta$ is ample on $S$ so $3(A+B+\Delta)$ is very ample (Theorem in Section 17 \cite{Mum}).
Thus by Bertini's theorem (\cite[Theorem II.8.18 and Remark III.7.9]{Ha})
there exists an integral and nonsingular curve $C$ on $S$ such that $C\sim 3(A+B+\Delta)$. Let $\mathcal I_C$ be the ideal sheaf of $C$ in $X$. Let $\tau:Y\rightarrow X$ be the blow up of $C$, so that $Y=\mbox{Proj}(\oplus_{n\ge 0}\mathcal I_C^n)$. Let $F=\mbox{Proj}(\oplus_{n\ge 0}\mathcal I_C^n/\mathcal I_C^{n+1})$ be the exceptional divisor and let $\overline\tau:F\rightarrow C$ be the induced morphism.  

The composed morphism $Y\rightarrow \overline Z$ contracts $F$ and the strict transform $\overline S$ of $S$ to the point $q$ of $\overline Z$ and is an isomorphism everywhere else, and is the blow up of an ideal sheaf $\mathcal I$ of $\mathcal O_Z$ such that the support of $\mathcal O_Z/\mathcal I$ is the point $q$.  The map $\tau$ induces an isomorphism of $\overline S$ and $S$. 

Let $G=\pi^*(C)$ which is an integral nonsingular surface in $X$. We have that $C$ is the scheme theoretic intersection of $G$ and $S$. Thus 
$$
(\mathcal O_X(-S)\otimes\mathcal O_C)\oplus (\mathcal O_X(-G)\otimes\mathcal O_C)\cong \mathcal I_C/\mathcal I_C^2,
$$
and so $F=\PP(\mathcal O_X(-S)\otimes\mathcal O_C)\oplus (\mathcal O_X(-G)\otimes\mathcal O_C))$.
We have that $\mathcal O_Y(-F)=\mathcal I_C\mathcal O_Y=\mathcal O_Y(1)$, so $\mathcal O_Y(-F)\otimes\mathcal O_F\cong \mathcal O_F(1)$. 

The Picard group of $F$ is
$$
\mbox{Pic}(F)=\mathcal O_F(1)\ZZ\oplus \overline\tau^*\mbox{Pic}(C).
$$

A canonical divisor of $Y$ is $K_Y=\tau^*K_X+F=-2\overline S-F+(\pi\tau)^*(-3(A+2B+3\Delta))$.

Since $\tau^*(S)=\overline S+F$ and $C$ is (isomorphic) to the scheme theoretic intersection of $F$ and $\overline S$, 
we have that 
$$
\mathcal O_Y(\overline S)\otimes\mathcal O_{\overline S}\cong \mathcal O_S(-3(2A+3B+4\Delta)).
$$
We have that 
$$
\mathcal O_Y(-n\overline S-jF)\otimes\mathcal O_{\overline S}\cong \mathcal O_X(-nS)\otimes\mathcal O_S(-(j-n)C)
$$
so by (\ref{eq1}),
\begin{equation}\label{eq2}
\mathcal O_Y(-n\overline S-jF)\otimes\mathcal O_{\overline S}\in \overline{\mbox{Eff}}(\overline S)
\mbox{ if and only if } j<n\left(3-\frac{\sqrt{3}}{3}\right).
\end{equation}

We have that $F\cong \PP(\mathcal E)$ where $\mathcal E=\mathcal O_C\oplus (\mathcal O_X(H)\otimes\mathcal O_C)$ with $H=S-G$ (by \cite[Proposition V.2.2]{Ha}). Let $C_0$ be the section of $\overline\tau$ corresponding to the natural surjection
of $\mathcal O_C$-modules $\mathcal E\rightarrow \mathcal O_X(H)\otimes\mathcal O_C$. Then $\mathcal O_F(C_0)\cong \mathcal O_{\PP(\mathcal E)}(1)$ and since 
$$
\mathcal O_X(H)\otimes\mathcal O_C\cong \mathcal O_S(-3(2A+3B+4\Delta))\otimes\mathcal O_C,
$$
we have that
$$
(C_0^2)_F=\deg \mathcal O_X(H)\otimes\mathcal O_C=(3(A+B+\Delta)\cdot3(-2A-3B-4\Delta))_S
=-9\times18=-162.
$$
Let $f$ be a fiber of $\overline\tau$ over a closed point of $C$. 

Let $K_F$ be a canonical divisor of $F$. By adjunction, we have
\begin{equation}\label{eq3}
\mathcal O_F(K_F)\cong \mathcal O_Y(K_Y+F)\otimes \mathcal O_F\cong\mathcal O_Y(-2\overline S+(\pi\tau)^*(-3(A+2B+3\Delta)))\otimes\mathcal O_F.
\end{equation}
 By adjunction on $f$ and $C_0$, 
 \begin{equation}\label{eq4}
 \mathcal O_F(K_F)\cong\mathcal O_F(-2C_0)\otimes\overline\tau^*(\mathcal O_C(K_C)\otimes\mathcal O_X(H)).
 \end{equation}
Let $C_{\overline S}$ be the scheme theoretic intersection of $\overline S$ and $F$, which is an integral curve which is a section over $\overline \tau$.
Comparing (\ref{eq3}) and (\ref{eq4}), we have that 
$$
\mathcal O_F(-2C_{\overline S})\cong \mathcal O_F(-2C_0)\otimes
\overline\tau^*(\mathcal O_C(K_C)\otimes \mathcal O_X(H)\otimes\mathcal O_S(3(A+2B+3\Delta)))
\cong \mathcal O_F(-2C_0)
$$
since $\mathcal O_C(K_C)\cong \mathcal O_X(G)\otimes\mathcal O_C$ by adjunction, as $\mathcal O_X(G)\otimes\mathcal O_S\cong \mathcal O_S(C)$. Since $(C_0^2)_F<0$, we have that $C_0=C_{\overline S}$.

We have that 
$$
\overline \tau^*(\mathcal O_X(S)\otimes \mathcal O_C)\cong \tau^*\mathcal O_X(S)\otimes\mathcal O_F
\cong \mathcal O_Y(\overline S+F)\otimes\mathcal O_F\cong \mathcal O_F(C_{\overline S})\otimes\mathcal O_Y(F).
$$
Thus 
$$
\mathcal O_Y(F)\otimes\mathcal O_F\cong \mathcal O_F(-C_0)\otimes\overline\tau^*(\mathcal O_X(S)\otimes\mathcal O_C),
$$
where 
$$
\begin{array}{lll}
\deg(\mathcal O_X(S)\otimes\mathcal O_C)&=&\deg \mathcal O_S(-3(A+2B+3\Delta))\otimes\mathcal O_C\\
&=&(-3(A+2B+3\Delta)\cdot 3(A+B+\Delta))\\
&=&-9\times12=-108.
\end{array}
$$
Thus $\mathcal O_Y(F)\otimes\mathcal O_F$ is represented in $(\mbox{Pic}(F)/\equiv)\otimes\RR$ by the class of 
$-C_0-108f$.

Suppose that $\gamma$ is an irreducible curve on $F$ which is not equal to $C_0$ and is not equal to a fiber over a closed point of $C$. There exists $n\in\ZZ$ and a divisor $\delta$ on $C$ such that $\gamma\sim nC_0+\overline\tau^*(\delta)$. Then $(\gamma\cdot f)>0$ implies $n>0$ and $(\gamma\cdot C_0)\ge 0$ implies $n(C_0^2)+\deg\delta\ge 0$. 
Thus $\deg \delta\ge -n(C_0)^2=n162>0$.

We now compute
$$
(\gamma^2)=n^2(C_0^2)+2n\deg\delta\ge n^2(C_0^2)-2n^2(C_0^2)=-n^2(C_0^2)>0.
$$
Thus $C_0$ is the only irreducible curve on $F$ with negative intersection number. 

It follows that 
  $\overline{\mbox{Eff}}(F)=\RR_+C_0+\RR_+f$ and 
  $$
  \overline{\mbox{Nef}}(F)=\{nC_0+mf\mid n,m\ge 0\mbox{ and }m\ge 162 n\}=
  \RR_+(C_0+162f)+\RR_+f.
  $$
  
  Let $j=n\left(3-\frac{\sqrt{3}}{3}\right)$. On $F$, we have the numerical equivalence
  $$
  (-n\overline S-jF)\cdot F\equiv (j-n)C_0+108jf
  =n \left(2-\frac{\sqrt{3}}{3}\right) C_0+108n \left(3-\frac{\sqrt{3}}{3}\right)f.
  $$
Letting $a=n \left(2-\frac{\sqrt{3}}{3}\right)$ and $b=108n \left(3-\frac{\sqrt{3}}{3}\right)$, we have that
$$
\frac{b}{a}=\frac{108}{33}\left(51+3\sqrt{3}\right)>162.
$$

Now suppose that $j=n$. Then 
$$
(-n\overline S-jF)\cdot F \equiv 108jf.
  $$
  
 The nature  of  the sections of $\mathcal O_Y(-n\overline S-jF)$ for $n,j\in\NN$ is determined by which of three separate regions of the positive quadrant of the plane contains the point $(n,j)$.  They are:
  \begin{enumerate}
  \item[1)] $j<n$
  \item[2)] $n\le j<n \left(3-\frac{\sqrt{3}}{3}\right)$ and
  \item[3)]   $n \left(3-\frac{\sqrt{3}}{3}\right)<j$.
  \end{enumerate}
    
  In case 1), 
  \begin{equation}\label{eq14}
  \mathcal O_Y(-n\overline S-jF)\otimes\mathcal O_F\not\in \overline{\mbox{Eff}}(F).
  \end{equation}
  In case 2), 
  \begin{equation}\label{eq15}
  \mathcal O_Y(-n\overline S-jF)\otimes\mathcal O_{\overline S}\in \overline{\mbox{Nef}}(\overline S)\mbox{ and }
  \mathcal O_Y(-n\overline S-jF)\otimes\mathcal O_F \in \overline{\mbox{Nef}}(F).
  \end{equation}
  In case 3), 
  \begin{equation}\label{eq16}
  \mathcal O_Y(-n\overline S-jF)\otimes\mathcal O_{\overline S}\not\in \overline{\mbox{Eff}}(\overline S).
  \end{equation}
  
  Let $R=\mathcal O_{\overline Z,q}$ ($q$ is the point on $\overline Z$ which $\overline S$ and $F$ contract to) and $U=Y\times_{\overline Z}\mbox{Spec}(R)$ with the natural projective morphism $U\rightarrow \mbox{Spec}(R)$ induced by $Y\rightarrow \overline Z$. The morphism $U\rightarrow\mbox{Spec}(R)$ is the blow up of the $m$-primary ideal $\mathcal I_q$. An effective Cartier divisor $D=n\overline S+jF$ on $U$ is anti-nef on $U$ if and only if $\mathcal O_Y(-D)\otimes\mathcal O_{\overline S}\in \overline{\mbox{Nef}}(\overline S)$ and $\mathcal O_Y(-D)\otimes\mathcal O_F\in\overline{\mbox{Nef}}(F)$. If $D$ is anti-nef on $U$, then  $\gamma_{\overline S}(D)=n$ and $\gamma_F(D)=j$ by Lemma \ref{Lemma3}.
   
  We deduce the following theorem.
  
  \begin{Theorem}\label{Theorem1} Let $D=n\overline S+jF$ with $j,n\in \NN$, an effective exceptional  divisor on $U$.
  \begin{enumerate}
  \item[1)] Suppose that $j<n$. Then  $\gamma_{\overline S}(D)=n$ and $\gamma_F(D)=n$.
  \item[2)]   Suppose that $n\le j<n \left(3-\frac{\sqrt{3}}{3}\right)$. Then   $\gamma_{\overline S}(D)=n$ and $\gamma_F(D)=j$.
  \item[3)] Suppose that $n \left(3-\frac{\sqrt{3}}{3}\right)<j$. Then   $\gamma_{\overline S}(D)=\frac{3}{9-\sqrt{3}}j$ and $\gamma_F(D)=j$.
  \end{enumerate}
  In all three cases, $-\gamma_{\overline S}(D)\overline S-\gamma_F(D)F$ is nef on $U$. 
    \end{Theorem}
    
    \begin{proof} If $D$ is in case 1), then $\Gamma(U,\mathcal O_U(-nm(\overline S+F)))=\Gamma(U,\mathcal O_U(-mD))$ for all $m\in \NN$ by (\ref{eq14}) and $\overline S+F$ is anti-nef on $U$ by (\ref{eq15}). If $D$ is in case 2), then $D$ is anti-nef on $U$ by (\ref{eq15}). If $D$ is in case 3), then 
    $\Gamma(U,\mathcal O_U(-\lceil \frac{m3j}{9-\sqrt{3}}\overline S+mjF\rceil))=\Gamma(U,\mathcal O_U(-mD))$ for all $m\in\NN$ by (\ref{eq16}) and $\frac{3}{9-\sqrt{3}}\overline S+F$ is anti-nef on $U$ by (\ref{eq15}).
    \end{proof}

  From Theorem \ref{Theorem1}, (\ref{eq11}) and Proposition \ref{Prop1}, we have that 
   for $D=n\overline S+jF$ with $m,j\in\NN$,
   \begin{equation}\label{eq12}
  \lim_{m\rightarrow\infty}
  \frac{\ell_R(R/I(mD)}{m^3}=-\frac{((-\gamma_{\overline S}(D)\overline S-\gamma_F(D)F)^3)}{3!}, 
  \end{equation}
  and  we have by (\ref{eqV6}), Theorem \ref{Theorem1}, (\ref{eq10}) and Proposition \ref{Prop1} that 
  \begin{equation}\label{eq13}
  \begin{array}{l}
  \lim_{m\rightarrow\infty}\frac{\ell_R(R/I(mn\overline S)I(mjF))}{m^3}
   =\sum_{i_1+i_2=3}\frac{1}{i_1!i_2!}e_R(\mathcal I(1)^{[i_1]},\mathcal I(2)^{[i_2]})
   n^{i_1}j^{i_2}\\
  =\sum_{i_1+i_2=3}\frac{-1}{i_1!i_2!}\left((-\gamma_{\overline S}(\overline S)\overline S-\gamma_F(\overline S)F)^{i_1}\cdot(-\gamma_{\overline S}(F)\overline S-\gamma_F(F)F\right)^{i_2})n^{i_1}j^{i_2}\\
  =\sum_{i_1+i_2=3}\frac{-1}{i_1!i_2!}\left(\left(-\overline S-F\right)^{i_1}\cdot\left(-\frac{3}{9-\sqrt{3}}\overline S-F\right)^{i_2}
  \right)n^{i_1}j^{i_2}.
  \end{array}
    \end{equation}
    
    We now make equations (\ref{eq12}) and (\ref{eq13}) explicit.     
   To compute the necessary intersection numbers, we use the facts that
   $$
   \mathcal O_Y(\overline S)\otimes\mathcal O_{\overline S}\cong \mathcal O_S(-3(2A+3B+4\Delta)),\,\,
   \mathcal O_Y(F)\otimes\mathcal O_{\overline S}\cong \mathcal O_S(3(A+B+\Delta)),\,\,
   $$
   $$
   \mathcal O_Y(\overline S)\otimes\mathcal O_F\cong \mathcal O_F(C_0),\,\,
   \mathcal O_Y(F)\otimes\mathcal O_F\equiv \mathcal O_F(-C_0-108f),
   $$
  to calculate that
  $$
  \begin{array}{l}
  (\overline S^3)=(-3(2A+3B+4\Delta)\cdot -3(2A+3B+4\Delta))_S=9\times52=468\\
  (\overline S^2\cdot F)=(-3(2A+3B+4\Delta)\cdot 3(A+B+\Delta))_S=-9\times18=-162\\
  (\overline S\cdot F^2)=(3(A+B+\Delta)\cdot 3(A+B+\Delta))_S=9\times 6=54\\
  (F^3)=((-C_0-108f)\cdot(-C_0-108f))_F=9\times6 =54.
  \end{array}
  $$
  
 The formulas of  Theorems \ref{Theorem2} and \ref{Theorem3} of the introduction are now a consequence of Theorem \ref{Theorem1}, equations (\ref{eq12}) and (\ref{eq13}) and the above formulas computing intersection multiplicities.

\end{document}